%% file: besicovitch-boundary.tex
\documentclass{article}
\input{preamble}

\title{A Vitali-type lemma for the boundary}
\author{Julian Weigt\footnote{
University of Warwick,
Mathematics Institute,
Zeeman Building,
Coventry CV4 7AL,
United Kingdom,
\texttt{julian.weigt@warwick.ac.uk}
}}

\begin{document}
\maketitle

\begin{abstract}
Take a set of balls in $\mathbb{R} ^d$.
We find a subset of pairwise disjoint balls whose combined perimeter controls the perimeter of the union of the original balls.
This can be seen as a boundary version of the Vitali covering lemma. 

We further prove combined volume-perimeter results and counterexamples and apply them to find short proofs of some regularity statements for maximal functions.
\end{abstract}

\begingroup
\begin{NoHyper}%
\renewcommand\thefootnote{}\footnotetext{%
2020 \textit{Mathematics Subject Classification.} 28A75, 42B25.\\%
\textit{Key words and phrases.} Vitali Covering Lemma, Besicovitch Covering Theorem, boundary, perimeter.\\%
This work was supported by the European Union's Horizon 2020 research and innovation programme (Grant agreement No. 948021).%
}%
\addtocounter{footnote}{-1}%
\end{NoHyper}%
\endgroup

\section{Introduction and main results}

The Vitali covering lemma is a classical tool used frequently in analysis.

\begin{theorem}[Vitali covering lemma]
\label{theorem_vitali_covering}
Let $\B$ be a set of balls $B\subset \mathbb{R} ^d$ with \(\sup\{\diam(B):B\in \B\}<\infty \).
Then there exists a subset $\S\subset \B$ of pairwise disjoint balls with
\[
\bigcup\B
\subset 
\bigcup\{5B:B\in \S\}
.
\]
\end{theorem}

The set $\S$ in the Vitali covering theorem can be seen as representing the Lebesgue measure $\lm{\bigcup\B}$ of $\bigcup\B$ well in the sense that
\[
\lmb{\bigcup\S}
\leq
\lmb{\bigcup\B}
\leq5^d
\sum_{B\in \S}
\lm B
=5^d
\lmb{\bigcup\S}
.
\]
We are interested in finding a set that represents $\sm{\mb{\bigcup\B}}$, the perimeter of $\bigcup\B$, well in a similar way.

Here, $\lmo$ denotes the $d$-dimensional Lebesgue measure, $\smo$ denotes the $d-1$-dimensional Hausdorff measure.
For a measurable set $E\subset \mathbb{R} ^d$, $\mb E$ denotes its measure theoretic boundary, defined by
\begin{align*}
\mc E
&\seq
\Bigl\{
x\in \mathbb{R} ^d
:
\limsup_{r\rightarrow 0}
\f{
\lm{B(x,r)\cap E}
}{
r^d
}
>
0
\Bigr\}
,&
\mb E
&\seq
\mc E
\cap 
\mc{\mathbb{R} ^d\setminus E}
.
\end{align*}
Moreover, for any domain $X$ and functions $a,b:X\rightarrow \mathbb{R} $ we write $a\lesssim b$ if there is a constant $C\geq0$ such that for all $\omega \in  X$ we have $a(\omega )\leq Cb(\omega )$.
We write $a\lesssim_{\omega _i} b$ if $C$ depends on a component $\omega _i$ of the argument $\omega $.
Most of the bounds in this manuscript will depend on the dimension $d$ of the underlying Euclidean space $\mathbb{R} ^d$.
Balls and intervals can be open or closed.
We will explicitly use the interior $\ti B$ or closure $\tc B$ of a ball where relevant.

Observe that the set $\S$ from the Vitali covering theorem does not always represent the perimeter well:
Consider a large ball $B_0$ surrounded by a large number of disjoint tiny balls, such that the combined perimeter of the tiny balls is much larger than the perimeter of the large ball, while the tiny balls are all contained in $2B_0$ as in \cref{fig_large_tiny_balls}.
\begin{figure}
\center
\includegraphics{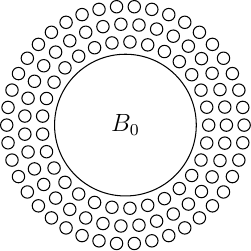}
\caption{A ball surrounded by many small balls with much larger combined perimeter.}
\label{fig_large_tiny_balls}
\end{figure}
The Vitali covering lemma would select $\S=\{B_0\}$ as the representative of that set of balls.
However, $B_0$ has a much smaller perimeter than the union of all balls, which means that
the disjoint subcollection of tiny balls would represent the perimeter of the whole set better.
Our main result finds such a disjoint subcollection that represents the perimeter well for any collection of balls.

\begin{theorem}
\label{pro_besikovitch_boundary}
Let $\B$ be a set of balls $B\subset \mathbb{R} ^d$ with \(\sup\{\diam(B):B\in \B\}<\infty \).
Then there is a subset $\S\subset \B$ of pairwise disjoint balls with
\begin{equation}
\label{eq_besicovitch_boundary}
\smb{\mb{\bigcup\B}}
\lesssim_d
\smb{\mb{\bigcup\S}}
.
\end{equation}
\end{theorem}

The set of balls $\S$ from \cref{pro_besikovitch_boundary} might not represent the volume of $\bigcup\B$ well.
Our next result finds a subcollection that represents both the volume and the perimeter well at the same time, at the expense of full disjointness.

\begin{theorem}
\label{theorem_vitali_boundary}
Let $\B$ be a set of balls $B\subset \mathbb{R} ^d$ with \(\sup\{\diam(B):B\in \B\}<\infty \).
Then for each $\varepsilon >0$ there exists a subset $\S\subset \B$ such that for any distinct $S_1,S_2\in \S$ we have
\[
\lm{S_1\cap S_2}
\leq
\varepsilon \min\{\lm{S_1},\lm{S_2}\}
,
\]
and with
\begin{align}
\label{eq_vitali_boundary_bound_global}
\lmb{\bigcup\B}
&\lesssim_d
\lmb{\bigcup\S}
,&
\smb{\mb{\bigcup\B}}
&\lesssim_d\varepsilon ^{-\f{d-1}{d+1}}
\smb{\mb{\bigcup\S}}
.
\end{align}
\end{theorem}

In order to achieve the rate $\varepsilon ^{-\f{d-1}{d+1}}$ in \cref{eq_vitali_boundary_bound_global} we use \cref{pro_besikovitch_boundary} in the proof of \cref{theorem_vitali_boundary}.
We also present a simpler proof that is independent of \cref{pro_besikovitch_boundary} but yields a worse rate.
The following \lcnamecref{example_muchboundary} shows that this rate $\varepsilon ^{-\f{d-1}{d+1}}$ in \cref{eq_vitali_boundary_bound_global} is optimal.

\begin{example}
\label{example_muchboundary}
For any sufficiently small $\varepsilon ,\delta >0$ there exists a collection of balls $\B$ such that each subcollection $\S\subset \B$ with
\begin{align*}
\lmb{\bigcup\B}
&\leq\delta ^{-1}
\sum_{B\in \S}
\lm B
,&
\smb{\mb{\bigcup\B}}
&\lesssim_d\varepsilon ^{-\f{d-1}{d+1}}
\sum_{B\in \S}\sm{\tb B}
\end{align*}
has cardinality at least 2 and a ball $B_0\in \S$ such that for all $B_1\in \S$ we have
\[
\lm{B_0\cap B_1}
\geq
\varepsilon \min\{\lm{B_0},\lm{B_1}\}
.
\]
\end{example}

We construct \cref{example_muchboundary} in \cref{sec_construction}.

For $d\geq2$ \cref{example_muchboundary} means in particular that we cannot find a disjoint subset of $\B$ that represents both the volume and the perimeter of $\bigcup\B$ well at the same time, i.e.\ it is not possible to strengthen \cref{pro_besikovitch_boundary} to also concern the volume.
In one dimension however this is possible: \Cref{theorem_vitali_boundary} holds with full disjointness.

\begin{theorem}
\label{theorem_vitali_boundary_oned}
Let $\I$ be a set of intervals $I\subset \mathbb{R} $ with \(\sup\{\diam(I):I\in \I\}<\infty \).
Then there exists a subset $\S\subset \I$ of intervals whose closures are pairwise disjoint and such that
\begin{align}
\label{eq_vitali_boundary_bound_oned_global}
\lmib{\bigcup\I}
\leq5
&\lmib{\bigcup\S}
,&
\smib{\mb{\bigcup\I}}
&\leq
\smib{\mb{\bigcup\S}}
.
\end{align}
\end{theorem}

This is consistent with the observation that for $d=1$ \cref{eq_vitali_boundary_bound_global} does not depend on $\varepsilon $, so it is reasonable to suggest that one could set $\varepsilon =0$.

The proofs of \cref{theorem_vitali_boundary,pro_besikovitch_boundary} both rely on the following main tool from \cite{zbMATH07523070}.

\begin{proposition}[{\cite[Proposition~4.3]{zbMATH07523070}} for a general set of balls]
\label{pro_levelsets_finite_g}
Let \(0<\lambda <1\).
Let \(E\subset \mathbb{R} ^d\) be a set of locally finite perimeter and let \(\B\) be a set of balls \(B\) with \(\lm{E\cap  B}>\lambda \lm B\).
Then
\[
\sm{\mb{\bigcup\B}\setminus\mc E}
\lesssim_d\lambda ^{-\f{d-1}d}(1-\log\lambda )
\smb{\mb E \cap \bigcup\B}
.
\]
\end{proposition}

We can bootstrap and remove the factor $(1-\log\lambda )$, which yields the optimal rate $\lambda ^{-\f{d-1}d}$, see \cref{pro_levelsets_finite_l_local_optimal}.

\section{Further variants}


\subsection{Local bounds}

The statement of the Vitali covering lemma, that for each $B\in \B$ there is an $S\in \S$ with $B\subset 5S$ is of a local nature.
Indeed, \cref{theorem_vitali_boundary,theorem_vitali_boundary_oned} are also true locally as follows:
The set $\B$ in \cref{theorem_vitali_boundary} has a cover $\{\B_S:S\in \S\}$ with \(\B=\bigcup_{S\in \S}\B_S\) such that for each $S\in \S$ we have
\begin{align}
\label{eq_vitali_boundary_bound}
\bigcup\B_S
&\subset 
\f{23}7S
,&
\smb{\mb{\bigcup\B_S}}
&\lesssim_d\varepsilon ^{-\f{d-1}{d+1}}
\sm{\tb S}
,
\end{align}
and similarly the set $\I$ in \cref{theorem_vitali_boundary_oned} has a cover $\{\I_S:S\in \S\}$ with \(\I=\bigcup_{S\in \S}\I_S\) such that for each $S\in \S$ we have
\begin{align}
\label{eq_vitali_boundary_bound_oned}
\bigcup\I_S
&\subset 
5S
,&
\smib{\mb{\bigcup\I_S}}
&\leq
\smi{\tb S}
.
\end{align}
We prove these results along with \cref{theorem_vitali_boundary,theorem_vitali_boundary_oned}.
The disjointness or near disjointness of $\S$ straightforwardly implies two-sided bounds between $\sum_{S\in \S}\lm S$ and $\lmb{\bigcup\S}$.

That means the local volumetric bounds in \cref{eq_vitali_boundary_bound,eq_vitali_boundary_bound_oned} straightforwardly imply the global volumetric bounds in \cref{eq_vitali_boundary_bound_global,eq_vitali_boundary_bound_oned_global}. 
The perimetric bounds in \cref{eq_vitali_boundary_bound,eq_vitali_boundary_bound_oned} however only imply \cref{eq_vitali_boundary_bound_global,eq_vitali_boundary_bound_oned_global} with the perimeter of the union on the right hand side replaced by a sum of perimeters, which is a weaker bound.
Note, that the estimate in \cref{example_muchboundary} proves the optimality of the rate $\varepsilon ^{-\f{d-1}{d+1}}$ even for that weaker bound.
On the other hand, for $d\geq2$ \cref{pro_besikovitch_boundary} does not hold with a stronger local conclusion.

\begin{example}
\label{example_localfails}
Let $d\geq2$.
Then for each $C>0$ there exists a set of balls $\B$ which has no subset $\S\subset \B$ of pairwise disjoint balls for which there is an associated cover $\{D_S:S\in \S\}$ of the boundary \(\mb\bigcup\B=\bigcup_{S\in \S}D_S\) that satisfies $D_S\subset CS$ and \(\sm{D_S}\leq C\sm{\tb S}\).
\end{example}

We construct \cref{example_localfails} in \cref{sec_construction}.

\subsection{Reverse inequalities}

In contrast to the volume in the Vitali covering lemma,
the reverse inequality to \cref{eq_besicovitch_boundary,eq_vitali_boundary_bound} is impossible to achieve in general:
For $\varepsilon >0$ consider 
\[
\B
=
\{B(x,1):x\in \mathbb{R} ^d,\ B(x,1)\cap  B(0,\varepsilon )=\emptyset \}
.
\]
Then
\[
\smb{\mb{\bigcup\B}}
=
\sm{\tb{B(0,\varepsilon )}}
\sim
\varepsilon ^{d-1}
\]
while for any $\S\subset \B$ set of disjoint balls we have
\[
\sum_{B\in \S}
\sm{\tb B}
\notin
(0,\sm{\tb{B(0,1)}})
.
\]
If we additionally assume that $\bigcup\B$ is bounded then for each $B\in \B$ we have
\[
\smb{\mb{\bigcup\B}}
\geq
\sm{\tb B}
,
\]
and thus the reverse inequality can be attained by removing sufficiently many balls from $\S$.

\subsection{Topological vs measure theoretic boundary}

For dimensions $d\geq2$ we may not replace the measure theoretic boundary by the topological boundary in \cref{pro_besikovitch_boundary,theorem_vitali_boundary}.
The reason is that there is a sequence of disjoint balls $B_1,B_2,\ldots $ with
\begin{align*}
\sum_{n=1}^\infty 
\lm{B_n}
&<\infty 
,&
\sum_{n=1}^\infty 
\sm{\tb{B_n}}
&<\infty 
,
\end{align*}
but such that $B_1\cup B_2\cup \ldots $ is topologically dense everywhere in $\mathbb{R} ^d$.
That means its topological boundary equals $\mathbb{R} ^d\setminus(B_1\cup B_2\cup \ldots )$ which even has infinite $d$-dimensional Hausdorff measure.

For $d=1$ a dense collection would always have
\[
\sum_{n=1}^\infty 
\sm{\tb{B_n}}
=
\sum_{n=1}^\infty 
2
=
\infty 
.
\]
Indeed, we could replace the measure theoretic boundary by the topological boundary for $d=1$ in \cref{pro_besikovitch_boundary,theorem_vitali_boundary,theorem_vitali_boundary_oned}.
For that we need to agree though that an "uncountable infinity", the $\smg^0$-measure of a set with positive Lebesgue measure, can be bounded by a "countable infinity", a countable sum of $2$s.
One small difference is that in \cref{theorem_vitali_boundary_oned} we would have to replace disjointness of the closures by just disjointness.

\section{Regularity of uncentered maximal functions}
\label{section_maximal}

In this \lcnamecref{section_maximal} we explain the origin of our results and applications to the regularity of maximal functions.
Denote by $\M$ the uncentered maximal operator, that is, given $f\in L^1_\loc(\mathbb{R} ^d)$ let
\begin{align*}
f_B
&=
\f1{\lm B}
\int _B|f|
,&
\M f(x)
&=
\sup_{\text{ball }B\ni x}
f_B
.
\end{align*}

\subsection{Superlevelsets}

One major application of the Vitali covering lemma is the proof of the endpoint $p=1$ of the Hardy-Littlewood maximal function theorem
\[
\|\M f\|_{1,\infty }
\leq 5^d
\|f\|_1
.
\]
For any $\lambda \in \mathbb{R} $ and $g:\mathbb{R} ^d\rightarrow \mathbb{R} $ we introduce the following abbreviation for the superlevelset:
\[
\{g>\lambda \}
\seq
\{x\in \mathbb{R} ^d:g(x)>\lambda \}
.
\]
We observe that the superlevelset of the uncentered maximal function equals the union of all balls on which the integral average of $|f|$ is larger than $\lambda $,
\begin{equation}
\label{eq_mfsuperlevelset}
\{\M f>\lambda \}
=
\bigcup\{B:
f_B>\lambda 
\}
.
\end{equation}
Applying the Vitali covering lemma to this set we obtain a pairwise disjoint set $\S$ of such balls such that
\[
\lm{\{\M f>\lambda \}}
\leq
5^n
\sum_{B\in \S}
\lm B
\leq
\f{
5^n
}\lambda 
\sum_{B\in \S}
\int _B|f|
\leq
\f{
5^n
}\lambda 
\|f\|_1
.
\]
The other endpoint bound \(\|\M f\|_\infty \leq\|f\|_\infty \) is straightforward to see.
By interpolation, for all $1<p\leq\infty $ one obtains
\[
\|\M f\|_p
\lesssim_{d,p}
\|f\|_p
.
\]

In \cite{MR1469106} Juha Kinnunen proved for all $1<p\leq\infty $ the corresponding gradient bound
\begin{equation}
\label{eq_gradient_mf_p}
\|\nabla \M f\|_p
\lesssim_{d,p}
\|\nabla f\|_p
.
\end{equation}
Since then there has been considerable interest in proving \cref{eq_gradient_mf_p} for $p=1$, or its relaxed version
\begin{equation}
\label{eq_var_mf}
\var(\M f)
\lesssim_d
\var(f)
.
\end{equation}
\Cref{eq_var_mf} was first proven in one dimension by Tanaka in 2002.
The best constant 1 for $d=1$ was later found by Aldaz and Pérez Lázaro.
\begin{theorem}[{\cite{zbMATH01774918,zbMATH05120644}}]
\label{theo_var_mf_oned}
For $d=1$ \cref{eq_var_mf} holds with constant $1$ for all $f:\mathbb{R} \rightarrow \mathbb{R} $ with $\var(f)<\infty $.
\end{theorem}

Less is known for higher dimensions.
In \cite{MR3800463} Luiro proved \cref{eq_var_mf} for radial functions for all $d\geq1$, and in \cite{MR3624402} Carneiro and Madrid used an estimate by Kinnunen and Saksman from \cite{MR1979008} to prove a variant of \cref{eq_var_mf} for certain fractional maximal operators in all dimensions.

A new approach to prove \cref{eq_var_mf} in higher dimensions is via the coarea formula,
\begin{equation}
\label{eq_coarea}
\var(f)
=
\int _{-\infty }^\infty 
\sm{\mb{
\{f>\lambda \}
}}
\intd \lambda 
.
\end{equation}
It has been employed by the author in \cite{zbMATH07523070,zbMATH07672859,weigt2024variation,weigt2021endpoint,beltran2021continuity} and others with some success, for example leading to the following result:

\begin{theorem}[{\cite{zbMATH07523070}}]
\label{theo_var_mf_charf}
Let $E\subset \mathbb{R} ^d$ and $f=\ind E$.
Then \cref{eq_var_mf} holds.
\end{theorem}

There, one is confronted with the perimeter of the superlevelset of the maximal function similarly to how the Lebesgue measure of the superlevelset occurs in the proof of the Hardy-Littlewood maximal function theorem.
Hence, for the case of the uncentered maximal function, we want to control also the perimeter of a union of balls.
This is the origin of the motivation for the results in this manuscript, and of its main tool, \cref{pro_levelsets_finite_g}.
More precisely, one of the main ideas of \cref{theorem_vitali_boundary} comes from \cite{zbMATH07523070,weigt2024variation} and is already heavily used in \cite[Section~5]{zbMATH07523070} and \cite[Proposition~3.20]{weigt2024variation}.

\Cref{pro_besikovitch_boundary} and the idea for its proof are new.
We also use it in the proof of \cref{theorem_vitali_boundary} to deduce the optimal rate in $\varepsilon $ in \cref{eq_vitali_boundary_bound}; the rate that would come out of the arguments from \cite{zbMATH07523070,weigt2024variation} is worse.

Much less about \cref{eq_var_mf} is known for the centered Hardy-Littlewood maximal operator which averages only over balls $B(x,r)$ which are centered in $x$.
The most important result is due to Kurka, who in \cite{MR3310075} proved \cref{eq_var_mf} for $d=1$ with a large constant.
Different centered operators are considered by Carneiro et.\ al.\ in \cite{MR3063097,MR3809456}, who proved \cref{eq_var_mf} in one dimension for convolution type maximal operators associated to a PDE.
For the centered Hardy-Littlewood maximal operator \cref{eq_mfsuperlevelset} only holds as an inclusion, and for convolution operators not even that.
Since one cannot deduce much about perimeters of sets from their inclusion properties, the approach via the coarea formula and the results from this manuscript appear unlikely to apply when proving regularity estimates for centered maximal operators.

More information on the progress addressing \cref{eq_var_mf} and related questions can be found and in many of the publications from the bibliography below, especially in \cite{carneiro2019regularity}.

\subsection{In one dimension}

We can also use \cref{pro_besikovitch_boundary} to give another short proof of \cref{theo_var_mf_oned}.
If we use the one-dimensional version \cref{theorem_vitali_boundary_oned} instead we also obtain the best constant 1.

First, we recall a basic fact about the measure theoretic boundary.

\begin{lemma}
\label{lem_betweenboundary}
Let $I$ be an interval and $E\subset \mathbb{R} $.
Then \(I\cap \mb E=\emptyset \) if and only if $I\subset \mi E$ or \(I\subset \mim{\mathbb{R} \setminus I}\).
\end{lemma}

Next we prove a sharp but easy one dimensional version of \cref{pro_levelsets_finite_g,pro_levelsets_finite_l_local_optimal}:

\begin{lemma}
\label{pro_levelsets_finite_l_local_optimal_1d}
For any interval $I$ and $E\subset \mathbb{R} $ with $\tc I\cap \mc E\neq\emptyset $ we have
\[
\smi{\tb I\setminus\mi E}
\leq
\smi{\tc I\cap \mb E}
.
\]
\end{lemma}

\begin{proof}
Denote by $a\leq b$ the endpoints of $I$.
It suffices to consider the case that $\tc I\cap \mb E$ is finite.
First consider the case that $\tc I\cap \mb E$ is empty.
Then by \cref{lem_betweenboundary} and \(\tc I\cap \mc E\neq\emptyset \) we must have \(\tb I\subset \tc I\subset \mi E\).

If $\tc I\cap \mb E$ is nonempty let $a'$ be its minimum and $b'$ its maximum.
If $a'<b'$ then
\[
\smi{\tb I\setminus\mi E}
\leq
2
=
\smi{\{a',b'\}}
\leq
\smi{\tc I\cap \mb E}
.
\]
If $a'=b'$ then by \cref{lem_betweenboundary} after possible reflection we have \([a,a')\subset \tim{\mathbb{R} \setminus I}\) and \((b',b]\subset \mi E\).
In conclusion
\[
\smi{\tb I\setminus\mi E}
=
\smi{\{a\}}
=
\smi{\{a'\}}
=
\smi{\tc I\cap \mb E}
.
\]
\end{proof}

Recall, that in any dimension for almost every $\lambda \in \mathbb{R} $ we have
\begin{equation}
\label{it_superlevelssame}
\lm{\{f\geq\lambda \}\setminus\{f>\lambda \}}
=0.
\end{equation}
Since the measure theoretic boundary is defined in terms of Lebesgue measure, this means that $\{f>\lambda \}$ and $\{f\geq\lambda \}$ have the same perimeter, and so the coarea formula also holds with the latter superlevelset in place of the former.

Next we collect a few properties of the intervals over which the maximal function averages.

\begin{lemma}
\label{lem_intervalssuperlevelsets}
Let $\lambda >0$ for which there exists an $x\in \mathbb{R} $ with $\M f(x)<\lambda $.
Let $\I_\lambda $ be the set of maximal intervals $I$ with $f_I=\lambda $.
Then the following holds:
\begin{enumerate}
\item
\label{it_superlevelssameintervals}
For every $\lambda \in \mathbb{R} $ we have
\[
\{\M f>\lambda \}
\subset 
\bigcup\I_\lambda 
\subset 
\{\M f\geq\lambda \}
.
\]
In particular, by \cref{it_superlevelssame} for almost every $\lambda \in \mathbb{R} $ all three sets have the same perimeter and so the coarea formula holds also with $\bigcup\I_\lambda $ instead of the superlevelset.
\item
\label{it_superlevelintersects}
Let $I\in \I_\lambda $.
Then \(\mb{\{f\geq\lambda \}}\) has infinite cardinality or intersects $\tc I$ in at least two points.
\end{enumerate}
\end{lemma}

\begin{proof}
The second inclusion of \cref{it_superlevelssameintervals} is straightforward to see.
For the first inclusion let $x\in \mathbb{R} $ with $\M f(x)>\lambda $.
Then there exists an interval $I\ni x$ with $f_I>\lambda $.
Let
\begin{align*}
a_x
&=
\inf\{a:\exists  b\ a<x<b,\ f_{(a,b)}\geq\lambda \}
,&
b_x
&=
\sup\{b:\exists  a\ a<x<b,\ f_{(a,b)}\geq\lambda \}
.
\end{align*}
Then $a_x<x<b_x$.
Assume $a_x=-\infty $.
Then for all $y< b_x$ we have $\M f(y)\geq\lambda $.
Let $y\geq b_x$ and $a<x<b$ with $f_{(a,b)}\geq\lambda $.
Then
\[
\M f(y)
\geq
\f1{y-a}
\int _a^y|f|
\geq
\f{b-a}{y-a}
\f1{b-a}
\int _a^b|f|
\geq
\f{b-a}{y-a}
\lambda 
.
\]
Since we can let $a\rightarrow -\infty $ while keeping $b$ bounded between $x$ and $y$ we can conclude $\M f(y)\geq\lambda $.
This is in contradiction to our assumption that there exists a $y\in \mathbb{R} $ with $\M f(y)<\lambda $.
That means $a_x>-\infty $ and similarly $b_x<\infty $.
By continuity of the integral that means there exists a maximal $I\ni x$ with $f_I\geq\lambda $.
It has to satisfy $f_I=\lambda $ because otherwise it could be further enlarged.

In order to prove \cref{it_superlevelintersects} observe that \(\lm{I\cap \{f\geq\lambda \}}>0\).
It suffices to consider the case that $\mb{\{f\geq\lambda \}}$ has finite cardinality.
That means $\{f\geq\lambda \}$ equals a finite union of closed intervals $J_1\cup \ldots \cup J_n$ up to Lebesgue measure zero.
By the maximality of $I$ there cannot be a $k\in \{1,\ldots n\}$ such that $\ti{J_k}$ contains an endpoint of $I$.
This implies
\(
\smi{\tb I\setminus\mia{\{f\geq\lambda \}}}
=
2
\)
and \cref{it_superlevelintersects} follows from \cref{pro_levelsets_finite_l_local_optimal_1d}.
\end{proof}

Recall the following approximation statement:

\begin{lemma}[{\cite[Theorem~5.2]{MR3409135}} for characteristic functions]\label{lem_l1approx}
Let \(\Omega \subset \mathbb{R} ^d\) be an open set and let \(A_1,A_2,\ldots \subset \mathbb{R} ^d\) with locally finite perimeter that converge to \(A\) in \(L^1_\loc(\Omega )\).
Then
\[\sm{\mb A\cap  \Omega }\leq\liminf_{n\rightarrow \infty }\sm{\mb A_n\cap  \Omega }.\]
\end{lemma}

\begin{proof}[Proof of \cref{theo_var_mf_oned}]
Let $\lambda >0$ for which there exists an $x\in \mathbb{R} $ with $\M f(x)<\lambda $, let $\I_\lambda $ be the set of maximal intervals $I$ with $f_I=\lambda $ and define
\(
\I_\lambda ^n
=
\{I\in \I_\lambda :
\lm I\leq n
\}.
\)
Apply \cref{theorem_vitali_boundary_oned} to $\I_\lambda ^n$ and let $\S_\lambda ^n$ be the resulting set.
Then by \cref{lem_intervalssuperlevelsets} \cref{it_superlevelintersects} we have
\[
\smib{\mb{
\bigcup\I_\lambda ^n
}}
\leq
\sum_{S\in \S_\lambda ^n}
\smi{\tb{
S
}}
\leq
\sum_{S\in \S_\lambda ^n}
\smi{
\tc S
\cap 
\mb{
\{f\geq\lambda \}
}}
\leq
\smi{
\mb{
\{f\geq\lambda \}
}}
.
\]
By \cref{lem_intervalssuperlevelsets} \cref{it_superlevelssameintervals} and \cref{lem_l1approx} for $\lmio$-almost every $\lambda \in \mathbb{R} $ we can conclude
\[
\smi{\mb{
\{\M f\geq\lambda \}
}}
=
\smib{\mb{
\bigcup\I_\lambda 
}}
\leq
\limsup_{n\rightarrow \infty }
\smib{\mb{
\bigcup\I_\lambda ^n
}}
\leq
\smi{
\mb{
\{f\geq\lambda \}
}}
.
\]
For $\lambda >0$ such that there exists no $x\in \mathbb{R} $ with $\M f(x)<\lambda $ the previous inequality holds as well since the left hand side is zero.
Integrating over $\lambda $ and applying the coarea formula \cref{eq_coarea} and \cref{it_superlevelssame} finishes the proof.
\end{proof}

\subsection{Characteristic functions in higher dimensions}

For higher dimensions this strategy fails because we cannot bound the perimeter of a ball $B$ with $f_B=\lambda $ by \(\sm{B\cap \mb{\{f\geq\lambda \}}}\).
One of the ideas in \cite{zbMATH07523070,zbMATH07672859,weigt2024variation} is to use also \(\sm{B\cap \mb{\{f\geq\mu \}}}\) for $\mu >\lambda $.
The simplest case is $f=\ind E$ where all superlevelsets of $f$ are the same.

Observe the following \lcnamecref{lem_boundaryofunion}:

\begin{lemma}[{\cite[Lemma~1.6]{zbMATH07523070}}]
\label{lem_boundaryofunion}
Let $A,B\subset \mathbb{R} ^d$ be measurable.
Then
\[
\mb{(A\cup B)}
\subset 
(\mb A\setminus\mc B)
\cup 
(\mb B\setminus\mc A)
\cup 
(\mb A\cap \mb B)
.
\]
In particular, if $\lm{B\setminus A}=0$ then
\[
\mb A
\subset 
(\mb A\setminus\mc B)
\cup 
\mb B
.
\]
\end{lemma}

Using that $\M f\geq f$ almost everywhere and \cref{lem_boundaryofunion}, in \cite{zbMATH07523070} we prove \cref{theo_var_mf_charf} as a consequence of \cref{pro_levelsets_finite_g}.
Here, it is crucial that \(\lambda ^{-\f{d-1}d}(1-\log\lambda )\) is integrable near $\lambda =0$.

\cite[Section~5]{zbMATH07523070} is dedicated to proving \cref{pro_levelsets_finite_l_local_optimal}, i.e.\ that \cref{pro_levelsets_finite_g} holds even without the factor $(1-\log\lambda )$.
With \cref{pro_besikovitch_boundary} at hand we can now run a drastically shorter proof.
We first show a slightly different variant.

\begin{proposition}[{\cite[Proposition~5.3]{zbMATH07523070} for a general set of balls}]
\label{pro_levelsets_finite_l_local}
Let \(0\leq\lambda<1/2\), let \(E\subset\mathbb{R}^d\) be a set of locally finite perimeter and let \(\B\) be a set of balls \(B\) with \(\lambda\lm B<\lm{E\cap B}\leq\f12\lm B\).
Then
\[\sm{\mb{\bigcup\B} }\lesssim_d\lambda^{-\f{d-1}d}\sm{\mb E \cap\bigcup\B }.\]
\end{proposition}

One of the key tools in \cite{zbMATH07523070} and here is the following relative isoperimetric inequality.

\begin{lemma}[{\cite[Theorem~5.11]{MR3409135}}]
\label{lem_isoperimetric}
Let $B$ be a ball and let $E\subset \mathbb{R} ^d$ be measurable.
Then
\[
\min\{\lm{B\cap E},\lm{B\setminus E}\}^{d-1}
\lesssim_d
\sm{\ti B\cap \mb E}^d
.
\]
\end{lemma}

\begin{proof}[Proof of \cref{pro_levelsets_finite_l_local}]
Let \(\S\) be the set of disjoint balls from \cref{pro_besikovitch_boundary}.
By the relative isoperimetric inequality \cref{lem_isoperimetric} for every $B\in \S$ we have
\[
\sm{\tb B}
\sim_d
\lm B^{\f{d-1}d}
<
\lambda ^{-\f{d-1}d}
\lm{B\cap  E}^{\f{d-1}d}
\lesssim_d
\lambda ^{-\f{d-1}d}
\sm{\mb E\cap  B}
\]
and we can conclude
\begin{align*}
\smb{
\mb{
\bigcup\B_\leq
}
}
&\lesssim_d
\sum_{B\in \S}
\sm{\tb B}
\lesssim_d
\lambda ^{-\f{d-1}d}
\sum_{B\in \S}
\sm{\mb E\cap  B}
=
\lambda ^{-\f{d-1}d}
\smb{
\mb E
\cap 
\bigcup\S
}
.
\end{align*}
\end{proof}

\begin{corollary}[{\cite[Proposition~5.3]{zbMATH07523070}} for a general set of balls]
\label{pro_levelsets_finite_l_local_optimal}
Let \(0<\lambda <1\).
Let \(E\subset \mathbb{R} ^d\) be a set of locally finite perimeter and let \(\B\) be a set of balls \(B\) with \(\lm{E\cap  B}>\lambda \lm B\).
Then
\[
\smb{\mb{\bigcup\B}\setminus\mc E}
\lesssim_d\lambda ^{-\f{d-1}d}
\smb{\mb E \cap \bigcup\B}
.
\]
\end{corollary}

\begin{proof}
Decompose $\B=\B_>\cup \B_\leq$ with
\[
\B_\lessgtr
=
\{B\in \B:\lm{B\cap  E}
\lessgtr
\lm B/2\}
.
\]
Then by \cref{lem_boundaryofunion} we have
\[
\smb{
\mb{
\bigcup\B
}
\setminus
E
}
\leq
\smb{
\mb{
\bigcup\B_>
}
\setminus
E
}
+
\smb{
\mb{
\bigcup\B_\leq
}
}
.
\]
By \cref{pro_levelsets_finite_g} we can bound the first summand and by \cref{pro_levelsets_finite_l_local} the second.
\end{proof}

In \cite{weigt2024variation} we prove \cref{eq_var_mf} for a maximal operator that averages over cubes.
A great deal of the manuscript is dedicated to dealing with overlapping cubes which suggests that \cref{pro_besikovitch_boundary} could be of use.
However, we have to consider many superlevelsets of the maximal function at the same time.
In that context the way of selecting balls from \cref{pro_besikovitch_boundary} is not robust enough and we use a strategy similar to the proof of \cref{theorem_vitali_boundary} instead.
Moreover, the incomplete disjointness from \cref{theorem_vitali_boundary} is not such an issue to deal with.
This is why, unfortunately, \cref{pro_besikovitch_boundary} does not seem to be useful in improving the results or simplifying the arguments in \cite{weigt2024variation}.

\section{Proofs of the main results}

The proofs of \cref{pro_besikovitch_boundary,theorem_vitali_boundary,theorem_vitali_boundary_oned} have a similar structure.
We present them in reverse order, which is the order of increasing difficulty.
For \cref{theorem_vitali_boundary,theorem_vitali_boundary_oned} we select balls in a similar fashion as the Vitali covering lemma, and for \cref{pro_besikovitch_boundary} we use the more complicated selection process from the Besicovitch covering lemma.
For \cref{pro_besikovitch_boundary,theorem_vitali_boundary} we then apply \cref{pro_levelsets_finite_g}.

\begin{remark}
The proofs of \cref{pro_besikovitch_boundary,theorem_vitali_boundary,theorem_vitali_boundary_oned} use induction over ordinals (transfinite induction).
Alternatively, one can run the exact same proofs using induction over the natural numbers in the case that the balls in $\B$ all belong to a bounded set, and then achieve the case of an unbounded set $\B$ by splitting $\mathbb{R} ^d$ into annuli with thickness equal to the diameter of the largest balls, as it is done in the proof of the Besicovitch Covering Theorem~1.27 in \cite{MR3409135}.
\end{remark}

\subsection{Proof of \texorpdfstring{\cref{theorem_vitali_boundary_oned}}{Theorem~1.5}}

\begin{proof}[Proof of \cref{theorem_vitali_boundary_oned}]
By transfinite induction select intervals $S_\alpha $ as follows:
Let $\alpha $ be an ordinal.
Let $\C$ be the set of intervals $I\in \I$ such that for all $\beta <\alpha $ we have
\(
\tc I
\cap 
\tc{S_\beta }
=
\emptyset 
.
\)
If $\C$ is empty, stop.
Otherwise choose $S_\alpha \in \C$ with
\[
\lmi{S_\alpha }
\geq
\f12
\sup\{
\lmi I:I\in \C
\}
\]
and set \(\I_{S_\alpha }=\{I\in \C:\tc I\cap \tc{S_\alpha }\neq\emptyset \}\).
Set \(\S=\{S_\alpha :\alpha \text{ ordinal}\}\).
Then for any $\alpha <\beta $ we have \(\tc{S_\alpha }\cap \tc{S_\beta }=\emptyset \).
Moreover, $\I=\bigcup_\alpha \I_{S_\alpha }$ and for all ordinals $\alpha $ we have \(\bigcup\I_{S_\alpha }\subset 5S_\alpha \).

The Lindelöf property states that every set of balls has a countable subset with the same union, see for example \cite[Proposition~1.5]{zbMATH05968626}.
That means it suffices to consider the case that $\I=\{I_1,I_2,\ldots \}$.
For any $n\geq1$ there exist open intervals $J^n_1,\ldots ,J^n_{k_n}$ with pairwise disjoint closures such that
\[
J^n_1\cup \ldots \cup J^n_{k_n}
\subset 
\tc{I_1}\cup \ldots \cup \tc{I_n}
\subset 
\tc{J^n_1}\cup \ldots \cup \tc{J^n_{k_n}}
.
\]
Moreover, each interval $\tc{I_i}$ intersects $\mc{\bigcup\S}$ by construction of $\S$ and thus so does each $\tc{J^n_i}$.
Thus, by \cref{pro_levelsets_finite_l_local_optimal_1d} for each $i=1,\ldots ,n_k$ we have
\begin{align*}
&
\smib{\tc{J^n_i}\cap \mbb{\bigcup\S\cup J^n_1\cup \ldots \cup J^n_{k_n}}}
=
\smib{\tc{J^n_i}\cap \mbb{\bigcup\S\cup J^n_i}}
\\
&\qquad\qquad=
\smib{\tc{J^n_i}\cap \tb{J^n_i}\setminus\mib{\bigcup\S}}
\leq
\smib{\tc{J^n_i}\cap \mb{\bigcup\S}}
.
\end{align*}
Since isolated points do not affect the measure theoretic boundary and using \cref{lem_boundaryofunion,lem_l1approx} we can conclude
\begin{align*}
\smib{\mb{\bigcup\I}}
&\leq
\limsup_{n\rightarrow \infty }
\smib{\mbb{\bigcup\S\cup J^n_1\cup \ldots \cup J^n_{k_n}}}
\\
&=
\limsup_{n\rightarrow \infty }
\smib{\mbb{\bigcup\S\cup J^n_1\cup \ldots \cup J^n_{k_n}}\setminus\tc{J^n_1\cup \ldots \cup J^n_{k_n}}}
\\
&\qquad\qquad+
\sum_{i=1}^{k_n}
\smib{\mbb{\bigcup\S\cup J^n_1\cup \ldots \cup J_{k_n}}\cap \tc{J^n_i}}
\\
&\leq
\limsup_{n\rightarrow \infty }
\smib{\mb{\bigcup\S}\setminus\mc{J^n_1\cup \ldots \cup J^n_{k_n}}}
+
\sum_{i=1}^k
\smib{\mb{\bigcup\S}\cap \tc{J^n_i}}
\\
&\leq
\smib{\mb{\bigcup\S}}
.
\end{align*}
\end{proof}

\begin{proof}[Proof of \cref{eq_vitali_boundary_bound_oned}]
For the local variant we could do the same calculation with $\I_{S_\alpha }$ instead of $\I$ and $\{S_\alpha \}$ instead of $\S$.
Alternatively, we just observe that $\bigcup\I_{S_\alpha }$ is an interval except that it might not contain the endpoints of $S_\alpha $ which implies
\[
\smib{\mb{\bigcup\I_{S_\alpha }}}
=
2
=
\smi{\tb{S_\alpha }}
.
\]
\end{proof}

\subsection{Proof of \texorpdfstring{\cref{theorem_vitali_boundary}}{Theorem~1.3}}
\label{sec_proof_vitali_boundary}

We first present a proof of \cref{theorem_vitali_boundary} with a worse rate in $\varepsilon $, which similar to the proof of \cref{theorem_vitali_boundary_oned}, except that we need to apply \cref{pro_levelsets_finite_g}.
To prove the full \cref{theorem_vitali_boundary} we need \cref{pro_besikovitch_boundary} and more careful geometric considerations.

Before we begin we we still owe generalizing \cref{pro_levelsets_finite_g} from finite to arbitrary sets of balls.

\begin{proof}[Proof of \cref{pro_levelsets_finite_g} for a general set of balls]
By the Lindelöf property it suffices to consider the case that $\B$ is countable, i.e.\ $\B=\{B_1,B_2,\ldots \}$.
Then $B_1\cup \ldots \cup B_n$ converges to $\bigcup\B$ locally in $L^1(\mathbb{R} ^d)$.
That means by \cref{pro_levelsets_finite_g} for a finite set of balls and by \cref{lem_l1approx} we can conclude
\begin{align*}
\smb{\mb{\bigcup\B}\setminus\mc E}
&\leq
\limsup_{n\rightarrow \infty }
\smb{\mb{(B_1\cup \ldots \cup B_n)}\setminus\mc E}
\\
&\lesssim_d
\lambda ^{-\f{d-1}d}(1-\log\lambda )
\limsup_{n\rightarrow \infty }
\sm{\mb E\cap (B_1\cup \ldots \cup B_n)}
\\
&\leq
\lambda ^{-\f{d-1}d}(1-\log\lambda )
\smb{\mb E\cap \bigcup\B}
.
\end{align*}
\end{proof}

\begin{proof}[Proof of \cref{theorem_vitali_boundary} with a suboptimal rate]
By transfinite induction select balls $B_\alpha $ as follows:
Let $\alpha $ be an ordinal.
Let $\C$ be the set of balls $B\in \B$ such that for all $\beta <\alpha $ we have
\(
\lm{B\cap B_\beta }
<
(\f78)^d
\varepsilon 
\lm B
.
\)
If $\C$ is empty, stop.
Otherwise choose $S_\alpha \in \C$ with
\[
r(S_\alpha )
\geq
\f78
\sup\{
r(B):B\in \C
\}
.
\]
Let $\B_\alpha $ be the set of those $B\in \B$ with
\(
\lm{B\cap S_\alpha }
\geq
(\f78)^d
\varepsilon 
\lm B
\)
and
\(
r(B)
\leq
\f87
r(S_\alpha )
.
\)

Set \(\S=\{S_\alpha :\alpha \text{ ordinal}\}\).
By construction, $\B=\bigcup_\alpha \B_\alpha $ and
\(
\bigcup\B_\alpha 
\subset 
\f{23}7
S_\alpha 
,
\)
proving the first parts of \cref{eq_vitali_boundary_bound_global,eq_vitali_boundary_bound}.
Let $\alpha <\beta $.
Then
\[
\lm{S_\beta \cap S_\alpha }
<
\Bigl(\f78\Bigr)^d
\varepsilon 
\lm{S_\beta }
\leq\varepsilon 
\lm{S_\alpha }
.
\]
For $\varepsilon $ small enough that means \(\f67S_\alpha \cap \f67S_\beta =\emptyset \).
Therefore,
\[
\lmb{\bigcup\B}
\leq
\sum_\alpha 
\lmb{
\f{23}7
S_\alpha 
}
\leq
\sum_\alpha 
\Bigl(\f{23}{6}\Bigr)^d
\lmb{
\f67
S_\alpha 
}
\leq
\Bigl(\f{23}{6}\Bigr)^d
\lmb{
\bigcup\S
}
.
\]

It remains to estimate the perimeter.
By \cref{lem_boundaryofunion} we have
\[
\smb{\mb{
\bigcup\B
}}
\leq
\smb{\mb{
\bigcup\B
\setminus
\mc{\bigcup\S}
}}
+
\smb{\mb{\bigcup\S}}
.
\]
To finish the proof of \cref{theorem_vitali_boundary} with a worse rate in $\varepsilon $ we may use that by \cref{pro_levelsets_finite_l_local_optimal} for each ordinal $\alpha $ we have
\[
\smb{\mb{
\bigcup\B
\setminus
\mc{\bigcup\S}
}}
\lesssim_d
\varepsilon ^{-\f{d-1}d}
\smb{\mb{\bigcup\S}}
\]
and stop.

Replacing $\B$ by $\B_\alpha $ and $\bigcup\S$ by $S_\alpha $ one can similarly prove the remaining part of \cref{eq_vitali_boundary_bound} with the same suboptimal rate.
\end{proof}

In order to finish the proof of \cref{theorem_vitali_boundary,eq_vitali_boundary_bound} it remains to improve the previous inequality:
For $\S$ constructed as above, we show that in fact
\begin{equation}
\label{eq_vitali_boundary_opt}
\smb{\mb{
\bigcup\B
\setminus
\mc{\bigcup\S}
}}
\lesssim_d
\varepsilon ^{-\f{d-1}{d+1}}
\smb{\mb{\bigcup\S}}
\end{equation}
holds, and so does the same bound with $\B$ replaced by $\B_\alpha $ and $\bigcup\S$ by $S_\alpha $ respectively.

The proof relies on the following slicing \lcnamecref{lem_fubini} for the boundary.

\begin{lemma}
\label{lem_fubini}
Let $E\subset \Omega \subset \mathbb{R} ^d$ measurable and for $y\in \mathbb{R} ^{d-1}$ denote $E_y=\{z\in \mathbb{R} :(y,z)\in E\}$.
Then
\[
\sm{\Omega \cap \mb E}
\geq
\int _{\mathbb{R} ^{d-1}}
\smi{
\Omega _y\cap 
\mb{
E_y
}}
\intd\lmg^{d-1}(y)
.
\]
\end{lemma}

\begin{proof}[Proof of \cref{eq_vitali_boundary_opt} and hence of \cref{theorem_vitali_boundary}]
Define
\begin{align*}
\B_{\geq}
&=
\Bigl\{B\in \B:
\lmb{B\cap \bigcup\S}
\geq\varepsilon ^{\f d{d+1}}
\lm B
\Bigr\}
,&
\B_{\alpha ,\geq}
&=
\Bigl\{B\in \B_\alpha :
\lm{B\cap S_\alpha }
\geq\varepsilon ^{\f d{d+1}}
\lm B
\Bigr\}
\end{align*}
and $B_{<}$, $B_{\alpha ,<}$ similarly so that
\[
\smb{\mb{
\bigcup\B
\setminus
\mc{\bigcup\S}
}}
\leq
\smb{\mb{
\bigcup\B_{\geq}
\setminus
\mc{\bigcup\S}
}}
+
\smb{\mb{
\bigcup\B_{<}
\setminus
\mc{\bigcup\S}
}}
\]
where by \cref{pro_levelsets_finite_l_local_optimal} we have
\[
\smb{\mb{
\bigcup\B_{\geq}
\setminus
\mc{\bigcup\S}
}}
\lesssim_d\varepsilon ^{-\f{d-1}{d+1}}
\smb{\mb{\bigcup\S}}
.
\]
The same bound with $\B_\geq$ replaced by $\B_{\alpha ,\geq}$ and $\bigcup\S$ by $S_\alpha $ respectively follows similarly.
It remains to control the contribution of $\B_{<}$ and $B_{\alpha ,<}$ respectively.
Apply \cref{pro_besikovitch_boundary} to $\B_{<}$ and denote by $\tilde{\B_{<}}$ the resulting set of pairwise disjoint balls.

For any $B\in \tilde{\B_{<}}$ there exists an $\alpha $ such that 
\[
\Bigl(\f78\Bigr)^d\varepsilon 
\lm B
\leq
\lm{B\cap S_\alpha }
\leq
\varepsilon ^{\f d{d+1}}
\lm B
\]
and let $\varepsilon '=\lm{B\cap S_\alpha }/\lm B$ so that \((\f78)^d\varepsilon \leq\varepsilon '\leq\varepsilon ^{\f d{d+1}}\).
We may approximate $B\cap S_\alpha $ by a union of two parabolic caps with curvatures $2/\rad B$ and $2/\rad{S_\alpha }\gtrsim2/\rad B$ respectively and equal radii.
For the moment rescale to $\rad B=1$ and denote by $t$ the radius of that parabolic cap with curvaturve $\rad B=1$ whose enclosed volume is $\varepsilon '\lm{B(0,1)}$.
Then
\begin{equation}
\label{eq_volume_to_parabola}
\varepsilon '
\sim_d
\varepsilon '\lm{B(0,1)}
\sim_d
t^{d-1}
t^2
=
t^{d+1}
.
\end{equation}
Scaling back to a general $\rad B$ we obtain that the parabolic caps approximating $B\cap S_\alpha $ have radius $t\rad B$.

\begin{figure}
\centering
\begin{minipage}{.6\textwidth}
\centering
\includegraphics{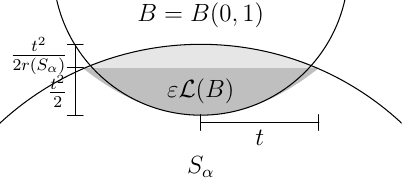}
\captionof{figure}{The intersection of two balls can be approximated by two parabolic caps.}
\end{minipage}%
\begin{minipage}{.4\textwidth}
\centering
\includegraphics{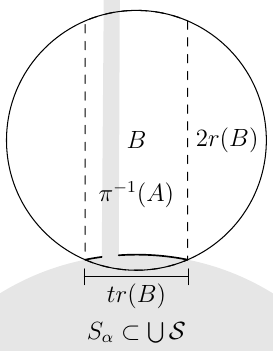}
\captionof{figure}{The preimage of the projection of the parabolic cap, $B$ and $\bigcup\S$ do not intersect much.}
\end{minipage}
\end{figure}

Denote by $A\subset \mathbb{R} ^{d-1}$ the projection $\pi $ of the parabolic cap along the line connecting the centers of $B$ and $S_\alpha $.
Then $A$ is a disc with radius $t\rad B$.
That means
\[
\lm{\pi ^{-1}(A)\cap B}
\sim_d
t^{d-1}
\lm B
\gtrsim_d
{\varepsilon '}^{\f{d-1}{d+1}}
\lm B
\gtrsim_d
{\varepsilon '}^{-\f{2}{d+1}}
\lmb{\bigcup\S\cap B}
\gtrsim_d
\varepsilon ^{-\f{2d}{(d+1)^2}}
\lmb{\bigcup\S\cap B}
.
\]
For $\varepsilon ^{\f{2d}{(d+1)^2}}\ll_d1$ this implies
\begin{align*}
&
\rad B
\lmg^{d-1}\Bigl(\Bigl\{y\in A:
\smg^1\Bigl(\pi ^{-1}(\{y\})\cap B\setminus\bigcup\S\Bigr)=0
\Bigr\}\Bigr)
\\
&\qquad\leq
\lmb{\bigcup\S\cap \pi ^{-1}(A)\cap B}
\leq
\f14
\lm{\pi ^{-1}(A)\cap B}
\leq
\f{
\rad B
}2
\lmg^{d-1}(A)
.
\end{align*}
Since by construction for every $y\in A$ we have
\[
\smg^1\Bigl(\pi ^{-1}(\{y\})\cap B\cap \bigcup\S\Bigr)>0
,
\]
that means for at least half the $y\in A$ that
\begin{equation}
\label{eq_varonmuchofcap}
0<
\f{
\lmib{\{y\}\times \mathbb{R} \cap B\cap \bigcup\S}
}{
\lmi{\{y\}\times \mathbb{R} \cap B}
}
<1
\end{equation}
so that by \cref{lem_fubini} we obtain
\begin{equation}
\label{eq_increaseboundary}
\smb{B\cap \mb{\bigcup\S}}
\geq
\f12
\lmg^{d-1}(A)
\gtrsim_d
(t\rad B)^{d-1}
\sim_d\varepsilon ^{\f{d-1}{d+1}}
\sm{\tb B}
.
\end{equation}
Since the balls in $\tilde{\B_{<}}$ are disjoint, so are their intersections with $\mb{\bigcup\S}$, and thus by \cref{lem_l1approx,pro_besikovitch_boundary} we obtain
\begin{align*}
\smb{\mb{\bigcup\B_{<}}\setminus\mc{\bigcup\S}}
&\lesssim_d
\sum_{B\in \tilde{\B_{<}}}
\sm{\tb B}
\lesssim_d\varepsilon ^{-\f{d-1}{d+1}}
\sum_{B\in \tilde{\B_{<}}}
\smb{B\cap \mb{\bigcup\S}}
\\
&\leq\varepsilon ^{-\f{d-1}{d+1}}
\smb{\mb{\bigcup\S}}
,
\end{align*}
finishing the proof of \cref{theorem_vitali_boundary}.

We can show the second part of \cref{eq_vitali_boundary_bound} by proving this inequality with $\B_<$ and $\bigcup\S$ replaced by $\B_{\alpha ,<}$ and $S_\alpha $ respectively in a similar way.
In fact, the proof simplifies significantly since in this case \cref{eq_varonmuchofcap} holds for all $y\in A$ without using any of the prior estimates.
\end{proof}

\subsection{Proof of \texorpdfstring{\cref{pro_besikovitch_boundary}}{Theorem~1.2}}

The outline of the proof of \cref{pro_besikovitch_boundary} is similar to the proof of \cref{theorem_vitali_boundary}, except that instead of a ball selection process similar to the Vitali covering lemma it uses the process of the Besicovitch covering lemma.
More precisely, it uses the statement of the Besicovitch covering lemma plus the extra fact that the cover can be chosen such that the center of each ball is covered by a not much smaller ball.

\begin{proposition}[Besicovitch covering lemma with extra fact]
\label{proposition_besicovitch}
There exists a constant $C_d\in \mathbb{N} $ depending only on the dimension $d$ such that the following holds.
Let $\B$ be a set of balls $B\subset \mathbb{R} ^d$ with \(\sup\{\diam(B):B\in \B\}<\infty \).
Then there exist subcollections $\S_1,\ldots ,\S_{C_d}\subset \B$ of pairwise disjoint balls such that each center $x$ of a ball $B(x,r)\in \B$ belongs to a ball $B(y,s)\in \S_1\cup \ldots \cup \S_{C_d}$ with $r\leq\f87s$.
\end{proposition}

The proof of the Besicovitch covering theorem from \cite{MR3409135} already yields this extra fact without recording it.
For the convenience of the reader we provide a proof here with full details.

\begin{lemma}
\label{lem_minangle}
Let \(x_0,x_1\in \mathbb{R} ^d\) and \(\f78\leq r_0\leq r_1\) with \(|x_1-x_0|\geq r_1-\f{r_0}7\) such that for \(i=0,1\) we have \(0\leq|x_i|-r_i<1\).
Moreover, assume $(|x_0|-\f32)(|x_1|-\f32)\geq0$.
Then the angle between \(x_0\) and \(x_1\) is at least $\arctan(1/7)$.
\end{lemma}

\begin{proof}
\begin{figure}
\begin{subfigure}{.5\textwidth}
\centering
\includegraphics{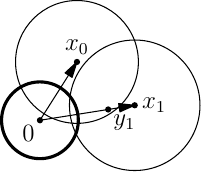}
\caption{Large balls.}
\end{subfigure}
\begin{subfigure}{.49\textwidth}
\centering
\includegraphics{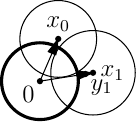}
\caption{Small balls.}
\end{subfigure}
\caption{Two balls that intersect a smaller ball such that none of the three includes anothers center must be separated by a certain angle.}
\label{fig_minangle}
\end{figure}
First consider $|x_0|,|x_1|\leq\f32$.
Then $|x_i|\geq r_i\geq\f78$ implies
\[
\bigl|
|x_1|-|x_0|
\bigr|
\leq
\f58
\leq
\f57r_1
\leq
r_1
-
\f27r_0
.
\]
Next consider $|x_0|,|x_1|\geq\f32$ which means
\(
|x_i|-1
\geq
\f{|x_i|}3
\geq
\f{r_i}3
.
\)
Take $i\in \{0,1\}$ such that \(|x_i|=\min\{|x_0|,|x_1|\}\).
Then
\[
\bigl|
|x_1|-|x_0|
\bigr|
\leq
1+r_{1-i}
-
|x_i|
\leq
r_{1-i}
-
\f{r_i}3
\leq
r_1
-
\f{r_0}3
.
\]
Define \(y_1=|x_0|\f{x_1}{|x_1|}\) so that \(|x_1-y_1|=\bigl||x_1|-|x_0|\bigr|\).
Then for both cases we can conclude
\[
|x_0-y_1|
\geq
|x_0-x_1|
-
|x_1-y_1|
\geq
r_1
-
\f{r_0}7
-
|x_1-y_1|
\geq
\f{r_0}7
\geq
\f{|x_0|}7
.
\]
Since $|y_1|=|x_0|$ this means the angle between $x_0$ and $y_1$ is at least $\arctan(1/7)$.
The angle between $x_0$ and $x_1$ is the same.
\end{proof}

\begin{corollary}
\label{corollary_intersectingballsbound}
For each dimension $d$ there is a $C_d$ such that the following holds:

Let $B(x,r)\subset \mathbb{R} ^d$ and let \(\B=\{B_\alpha =B(x_\alpha ,r_\alpha ):\alpha \}\) be an ordered set of balls such that for any $\alpha $ we have $r_\alpha \geq\f78 r$, $B(x,r)\cap B_\alpha \neq\emptyset $ and $x\notin B_\alpha $ and for any $\beta >\alpha $ we have $r_\beta \leq\f87r_\alpha $ and \(x_\beta \notin B_\alpha \).

Then the cardinality of $\B$ is bounded by $C_d-1$.
\end{corollary}

\begin{proof}
Rescale so that $x=0$, $r=1$ and let $\alpha <\beta $.
Then by assumption for $i=\alpha ,\beta $ we have $r_i\geq\f78$ and \(0\leq|x_i|-r_i\leq1\).
If $r_\alpha \geq r_\beta $ then clearly then \(r_\alpha \geq r_\alpha -\f{r_\beta }7\).
If $r_\alpha \leq r_\beta $ then still by assumption \(r_\alpha +\f{r_\alpha }7\geq r_\beta \).
In conclusion, for $\{i,j\}=\{\alpha ,\beta \}$ such that \(r_i=\min\{r_\alpha ,r_\beta \}\) we have
\[
|x_\alpha -x_\beta |
\geq
r_\alpha 
\geq
r_j
-
\f{r_i}7
.
\]
That means if $(|x_\alpha |-\f32)(|x_\beta |-\f32)\geq0$ then by \cref{lem_minangle} we can conclude that the angle between $x_\alpha $ and $x_\beta $ is bounded from below.
That means both sets
\(
\{\alpha :|x_\alpha |\geq\f32\}
\)
and
\(
\{\alpha :|x_\alpha |\leq\f32\}
\)
have uniformly bounded cardinality, and thus so has $\B$.
\end{proof}

\begin{proof}[Proof of \cref{proposition_besicovitch}]
By transfinite induction over ordinals $\alpha $ define the balls \(B_\alpha \) as follows:
Denote by $\B_\alpha $ the set of balls $B(x,r)\in \B$ for which there is no $\beta <\alpha $ with $x\in B_\beta $.
If $\B_\alpha $ is empty, stop.
Otherwise choose a ball \(B_\alpha \in \B_\alpha \) such that
\[
r(B_\alpha )
\geq\f78
\sup\{r(B):B\in \B_\alpha \}
.
\]

Then for each \(B(x,r)\in \B\) there is an ordinal $\alpha $ with $x\in B_\alpha =B(x_\alpha ,r_\alpha )$ and $r\leq\f87r_\alpha $.
For each ordinal $\alpha $ let $I_\alpha $ be the set of ordinals $\beta <\alpha $ with $B_\beta \cap B_\alpha \neq\emptyset $.
Then by \cref{corollary_intersectingballsbound} the cardinality of $I_\alpha $ is bounded by $C_d-1$.
Consequently, the function $\sigma $ on the ordinals defined transfinitely inductively by
\[
\sigma (\alpha )
=
\min\{n\leq \alpha :\forall  \beta \in  I_\alpha \ \sigma (\beta )\neq n\}
\]
is pointwise bounded by $C_d$.
By definition of $\sigma $ for every number \(n\leq C_d\) the balls in
\(
\{B_\alpha :\sigma (\alpha )=n\}
\)
are disjoint.
\end{proof}

\begin{proof}[Proof of \cref{pro_besikovitch_boundary}]
Let $\S_1,\ldots ,\S_{C_d}$ be the sets from \cref{proposition_besicovitch}.
For each $n=1,\ldots ,C_d$ and $S=B(y,s)\in \S_n$ denote by $\B_S$ the set of balls $B(x,r)\in \B$ with $x\in S$ and $r\leq\f87s$.
That means there is a ball $D\subset B(x,r)\cap S$ with radius $\f7{16}r$, and thus
\[
\lm{B(x,r)\cap S}
\geq
\Bigl(
\f7{16}
\Bigr)^d
\lm{B(x,r)}
.
\]
By \cref{proposition_besicovitch} for every $B\in \B$ there is an $n\in \{1,\ldots ,C_d\}$ and an $S\in \S_n$ with $B\in \B_S$.
Using \cref{lem_boundaryofunion,pro_levelsets_finite_g} we can conclude
\begin{align*}
\smb{\mb{\bigcup\B}}
&=
\smb{\mb{
\bigcup_{n=1}^{C_d}
\bigcup_{S\in \S_n}
\bigcup
\B_S
}}
\leq
\sum_{n=1}^{C_d}
\sum_{S\in \S_n}
\smb{\mb{
\bigcup
\B_S\setminus\tc S
}}
+
\sm{\tb S}
\\
&\lesssim_d
\sum_{n=1}^{C_d}
\sum_{S\in \S_n}
\sm{\tb S}
\leq
C_d
\sum_{S\in \S_m}
\sm{\tb S}
\end{align*}
for an \(m\in \{1,\ldots ,C_d\}\) for which
\(
\sum_{S\in \S_m}
\sm{\tb S}
\)
is maximal.
\end{proof}

\subsection{Construction of the \texorpdfstring{\lcnamecref{example_muchboundary}}{examples}}
\label{sec_construction}

In this \lcnamecref{sec_construction} we construct a single example which satisfies both the properties required in \cref{example_localfails,example_muchboundary}.

For $d=1$ no subset of \(\B=\{(0,1)\}\) satisfies the perimeter inequality in \cref{example_muchboundary} if we make the implicit constant smaller than $1$.
Hence it remains to consider $d=2$.

We define $\B$ inductively.
Given $n\in \mathbb{N} $, choose a ball $B_{n+1}$ with radius as large as possible but at most $\delta $ which is disjoint from $B_1,\ldots ,B_n$ and satisfies
\(
\lm{B(0,1)\cap B_{n+1}}
=\varepsilon 
\lm{B_{n+1}}
.
\)
Set $B_0=B(0,1)$ and \(\B=\{B_0,B_1,\ldots \}\).

\begin{figure}
\centering
\includegraphics{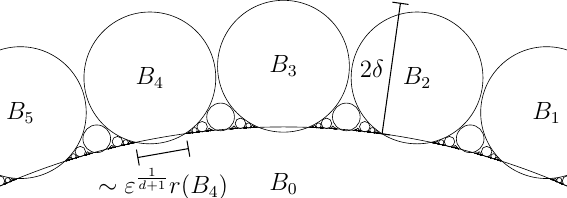}
\caption{Disjoint balls with arbitrarily large combined perimeter can be arranged intersecting a large ball.}
\end{figure}

We will show the following properties:
\begin{enumerate}
\item
\label{it_increaseboundary}
For each $n\in \mathbb{N} $
\[
\sm{B_n\cap \partial B_0}
\sim_d
\varepsilon ^{\f{d-1}{d+1}}
\sm{\partial B_n
\setminus
\tc{B_0}
}
,
\]
\item
\label{it_newboundary}
\[
\sm{
\partial B_n
\setminus
\tc{
\bigcup_{i\neq n}
B_i
}
}
=
\sm{
\partial B_n
\setminus
\tc{B_0}
}
,
\]
\item
\label{it_coverunitboundary}
\[
\smb{
\partial B_0\setminus(B_1\cup B_2\cup \ldots )
}
=
0
.
\]
\end{enumerate}

As a consequence,
\begin{align}
\notag
\smb{\mb{\bigcup_{n=0}^\infty B_n}}
&\geq
\smb{
\mb{
\bigcup_{n=1}^\infty 
B_n
}
\setminus
\tc{B_0}
}
\geq
\sum_{n=1}^\infty 
\smb{
\tb{
B_n
}
\setminus
\tc{
\bigcup_{i\neq n}
B_i
}
}
=
\sum_{n=1}^\infty 
\sm{
\tb{
B_n
}
\setminus
\tc{
B_0
}
}
\\
&\sim_d
\varepsilon ^{-\f{d-1}{d+1}}
\sum_{n=1}^\infty 
\sm{
B_n
\cap 
\partial B_0
}
\label{eq_largeboundary}
=
\varepsilon ^{-\f{d-1}{d+1}}
\sm{
\partial B_0
}
.
\end{align}

For any $n\geq1$ the ball $B_n$ belongs to the annulus $B(0,1+2\delta )\setminus B(0,1-2\delta )$.
Since $B_1,B_2,\ldots $ are pairwise disjoint this means
\begin{equation}
\label{eq_smallvolume}
\sum_{n=1}^\infty 
\lm{B_n}
\lesssim_d
\delta 
.
\end{equation}
Thus, for $\delta >0$ small enough any set $\S\subset \B$ that satisfies the assumptions of \cref{example_muchboundary} must contain $B_0$.
By \cref{eq_largeboundary} there must be additional balls from $\B$ in $\S$.
By construction they all satisfy $\lm{B_n\cap  B_0}\geq\varepsilon \lm{B_n}$.

For \cref{example_localfails} we consider $\tilde{\B}$ to be those $B\in \B$ that have a center $x$ in the right half space, i.e.\ with first coordinate $x_1\geq0$.
Then
\(
\sm{\mb{\bigcup\tilde{\B}}}
\gtrsim
\sm{\mb{\bigcup\B}}
\)
which means \cref{eq_smallvolume} also holds for $\tilde{\B}$ instead of $\B$.
It suffices to consider the only maximal subsets $\S$ of disjoint balls of $\tilde{\B}$, which are $\{B_0\}$ and $\tilde{\B}\setminus\{B_0\}$.
For $\varepsilon $ small enough, by \cref{eq_smallvolume} the ball $B_0$ has too small perimeter to bound the perimeter of $\bigcup\tilde{\B}$.
Since the balls $B\in \tilde{\B}\setminus\{B_0\}$ have their center in the right half space and radius at most $\delta $, for $\delta $ small enough they are too small for $CB$ to cover the part of $\tb{B_0}$ with first coordinate less than $-1/2$.

In order to finish proving the desired properties of \cref{example_localfails,example_muchboundary} it remains to show \cref{it_increaseboundary,it_newboundary,it_coverunitboundary}.

\begin{proof}[Proof of \cref{it_increaseboundary,it_newboundary,it_coverunitboundary}]
It follows like the proof of \cref{eq_vitali_boundary_bound} that \(\lm{B_i\cap B_0}=\varepsilon \lm{B_i}\) implies \cref{it_increaseboundary}, see \cref{eq_increaseboundary}.

Let $n\in \mathbb{N} $ and $\nu >0$.
Then, since $\rad{B_i}\rightarrow 0$ there is only a finite number of indices $i$ for which $B_i$ does not belong to $B(0,1+\nu )$.
That means $\smo$-almost every point of \(B(0,1+\nu )\cap \tb{B_n}\) belongs to $\mb{\bigcup\B}$.
Letting $\nu \rightarrow 0$ we obtain \cref{it_newboundary}.

It remains to prove \cref{it_coverunitboundary}.
For each $n$ let $s_n\in \partial B(0,1)$ be the radial projection of the center of $B_n$ to $\partial B(0,1)$, let $r_n$ such that $B(s_n,r_n)\cap \partial B(0,1)$ is the radial projection of $B_n$ to $\partial B(0,1)$ and
 let $\tilde r_n$ such that $B_n\cap \partial B(0,1)=B(s_n,\tilde r_n)\cap \partial B(0,1)$.
For $t\sim_d\varepsilon ^{\f1{d+1}}$ from \cref{eq_volume_to_parabola} we have $\tilde r_n\sim_d t\rad{B_n}$.
Note, that as $\rad{B_n}\leq \delta \rightarrow 0$ we have $r_n/\rad{B_n}\rightarrow 1$ and $\tilde r_n/(t\rad{B_n})\rightarrow 1$.

We will prove that for each $n\in \mathbb{N} $ and $\nu >0$ there exists an $m\in \mathbb{N} $ with
\begin{equation}
\label{eq_remainderblowupcover}
\partial B(0,1)
\setminus
\bigcup_{i=1}^n
B(s_i,(1+\nu )\tilde r_i)
\subset 
\bigcup_{i=n+1}^m
B(s_i,3r_i)
.
\end{equation}
Letting $\nu \rightarrow 0$ we obtain
\[
\smb{
\partial B(0,1)
\setminus
\bigcup_{i=1}^n
B(s_i,\tilde r_i)
}
\leq
\smb{
\bigcup_{i=n+1}^\infty 
\partial B(0,1)
\cap 
B(s_i,3r_i)
}
.
\]
Since for $i=1,2,\ldots $ the balls $B_i$ are pairwise disjoint, also their subsets $\partial B(0,1)\cap B(s_i,\tilde r_i)$ are, and hence
\[
\sum_{i=1}^\infty 
r_i^{d-1}
\lesssim_d
t^{1-d}
\sum_{i=1}^\infty 
\tilde r_i^{d-1}
\lesssim_d
t^{1-d}
\smb{
\bigcup_{i=1}^\infty 
\partial B(0,1)
\cap 
B(s_i,\tilde r_i)
}
\lesssim_d
t^{1-d}
.
\]
As a consequence,
\[
\smb{
\partial B(0,1)\setminus\bigcup_{n=1}^\infty B_n
}
\leq
\lim_{n\rightarrow \infty }
\smb{
\bigcup_{i=n+1}^\infty 
\partial B(0,1)
\cap 
B(s_i,3r_i)
}
\lesssim_d
\lim_{n\rightarrow \infty }
\sum_{i=n+1}^\infty 
r_i^{d-1}
=
0
.
\]

It remains to prove \cref{eq_remainderblowupcover}.
Assume that $\delta ,\nu >0$ are sufficiently small.
Then for each $1\leq i\leq n$ after rotation and shift, in a neighborhood of $s_i$ we may approximate $\tb{B(0,1)}$ by the plane \(\{x:x_d=0\}\) and \(\tb{B_i}\) by a parabola with curvature $1/r_i$ and zero set $\{x:|x-s_i|=\tilde r_i\}$.
\begin{figure}
\centering
\includegraphics{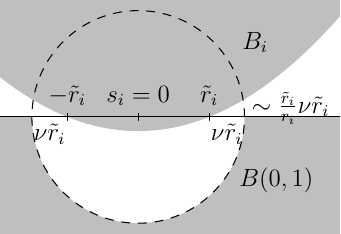}
\caption{The distance between the parabola \(x\mapsto \f{(x-\tilde r_i)(x+\tilde r_i)}{2r_i}\) and the lower half space without $B(0,(1+\nu )\tilde r_i)$ is \(\sim\frac{\tilde r_i}{r_i}\nu\tilde r_i\).}
\label{fig_plane_parabola}
\end{figure}
That means any two points in \(\partial B(0,1)\setminus B(s_i,(1+\nu )\tilde r_i)\) and $B_i$ have distance at least
\[
\sim
\f{\tilde r_i}{r_i}
\nu \tilde r_i
\sim
\nu t^2r_i
\geq
\nu t^2r_n
,
\]
see \cref{fig_plane_parabola}.

Take an $m\geq n$ with \(r_m\ll\nu t^2r_n\) sufficiently small.
Let $\B_n$ be the set of those balls $B(x,r)$ with \(\lm{B(x,r)\cap B(0,1)}=\varepsilon \lm B\) and \(r\geq r_m\) which are disjoint from $B_i$ for every $1\leq i\leq n$.
As a consequence of the above distance estimate for each
\[
s
\in 
\partial B(0,1)\setminus\bigcup_{i=1}^nB(s_i,(1+\nu )\tilde r_i)
\]
exists a $B(x,r)\in \B_n$ with $B(x,r)\ni s$.
By the way $\B$ is constructed, each $B(x,r)\in \B_n$ is intersected by a ball $B(x_k,r_k)\in \B_n\cap \B$ with \(r_k\geq r\).
That means $s\in B(x_k,3r_k)$.
Since \(r_k\geq r\geq r_m\) we have $k\leq m$ so that we can conclude \cref{eq_remainderblowupcover} which finishes the proof of \cref{it_coverunitboundary}.
\end{proof}

\printbibliography

\end{document}

%% file: preamble.tex
\usepackage{amsmath,amssymb,amsthm,hyperref,enumitem,esint,mathtools,subcaption,xcolor,cleveref}
\usepackage{geometry}

\usepackage[url=false,isbn=false,doi=false]{biblatex}
\theoremstyle{plain}
\newtheorem{theorem}{Theorem}[section]

\theoremstyle{definition}

\newtheorem{example}[theorem]{Example}
\newtheorem{lemma}[theorem]{Lemma}
\newtheorem{corollary}[theorem]{Corollary}
\newtheorem{proposition}[theorem]{Proposition}

\theoremstyle{remark}
\newtheorem{remark}[theorem]{Remark}

\allowdisplaybreaks

\numberwithin{equation}{section}
 
\newcommand\f\frac
\newcommand\intd{\mathop{}\!\mathrm{d}}

\newcommand\ind[1]{1_{#1}}
\newcommand\seq{:=}

\newcommand\loc{\mathrm{loc}}
\renewcommand\S{\mathcal S}

\newcommand\B{\mathcal B}
\newcommand\I{\mathcal I}
\newcommand\C{\mathcal C}
\DeclareMathOperator*\var{var}
\DeclareMathOperator\diam{diam}

\newcommand\M{{\mathrm M}}

\newcommand\tb[1]{\mathop{\partial{}}{#1}}
\newcommand\mb[1]{\mathop{\partial_*}{#1}}

\newcommand\mbb[1]{\mb{\Bigl(#1\Bigr)}}

\newcommand\tc[1]{\overline{#1}}
\newcommand\mc[1]{\tc{#1}^*}
\newcommand\ti[1]{\mathring{#1}}
\newcommand\tia[1]{#1^{\mathrm o}}
\newcommand\tim[1]{\tia{(#1)}}

\newcommand\mi[1]{\ti{#1}^*}
\newcommand\mia[1]{#1^{\mathrm o*}}
\newcommand\mim[1]{\mia{(#1)}}
\newcommand\mib[1]{\mia{\bigl(#1\bigr)}}
\newcommand\smg{\mathcal{H}}
\newcommand\smo{\smg^{d-1}}
\newcommand\sm[1]{\smg^{d-1}(#1)}
\newcommand\smb[1]{\smg^{d-1}\Bigl(#1\Bigr)}
\newcommand\smi[1]{\smg^0(#1)}
\newcommand\smib[1]{\smg^0\Bigl(#1\Bigr)}
\newcommand\lmg{\mathcal{L}}
\newcommand\lmo{\lmg^d}
\newcommand\lmio{\lmg^1}
\newcommand\lm[1]{\lmo(#1)}
\newcommand\lmi[1]{\lmio(#1)}
\newcommand\lmb[1]{\lmo\Bigl(#1\Bigr)}
\newcommand\lmib[1]{\lmio\Bigl(#1\Bigr)}
\newcommand\avint\fint
\newcommand\rad[1]{r(#1)}